\documentclass[reqno]{amsart}
\usepackage[psamsfonts]{amssymb}
\usepackage{amsthm}
\usepackage{color}

\newtheorem{theorem}{Theorem}
\newtheorem{pf}{Proof}
\newtheorem{lem}{Lemma}

\theoremstyle{remark}

\theoremstyle{definition}

 \theoremstyle{definition}
 
 \theoremstyle{remark}

 \numberwithin{equation}{section}

\def\bege{\begin{equation}} \def\ende{\end{equation}}

   \def\begr{\begin{eqnarray}}
\def\endr{\end{eqnarray}} 
 
\def\bege{\begin{equation}} \def\ende{\end{equation}}
\def\begr{\begin{eqnarray}} \def\endr{\end{eqnarray}}
\def\bnum{\begin{enumerate}} \def\enum{\end{enumerate}}

\allowdisplaybreaks[4]

\begin{document}
\title{\sf Compact and order bounded   sum of weighted  differentiation composition operators }


\author{Aakriti
 Sharma}

\address{Department of Mathematics,
Central University of Jammu,
Bagla,
Rahya-Suchani, Samba 181143,
INDIA }
\email{aakritishma321@gmail.com}

\author{Ajay K.~Sharma}
\address{Department of Mathematics,
Central University of Jammu,
Bagla,
Rahya-Suchani, Samba 181143,
INDIA }
\email{aksju\_76@yahoo.com}

\begin{abstract} In this paper, we characterize bounded, compact and order bounded   sum of weighted  differentiation composition operators  from Bergman type spaces to weighted Banach spaces of analytic functions, where
the   sum of weighted differentiation  composition  operators is defined  as $$ S^{n}_{\vec{u},\tau}(f)= \displaystyle\sum_{j=0}^{n}D_{u_{j}
,\tau}^{j}(f), \; \; f \in \mathcal{H}(\mathbb D).$$ Here $\mathcal{H}(\mathbb D)$ is the space of all holomorphic functions on $\mathbb D$, $\vec{u}=\{u_{j}\}_{j=0}^{n}$,  $u_{j} \in \mathcal{H}(\mathbb{D})$, $\tau$ a holomorphic self-map of $\mathbb D$, $f^{(j)}$ the $j$th derivative of $f$ and  weighted  differentiation composition operator  $D_{u_{j},\tau}^{j}$  is defined as $D_{u_{j},\tau}^{j}(f)=u_{j}C_{\tau}D^{j}(f)=u_{j}f^{(j)}\circ\tau, \; \; f \in \mathcal{H}(\mathbb D).$
\end{abstract}



\maketitle
\section{Introduction}  Let $\mathcal{H}(\mathbb{D})$ and $\mathcal{S}(\mathbb{D})$ be the class of holomorphic functions on $\mathbb{D}$ and the class of holomorphic self maps on $\mathbb{D}$ respectively, where $\mathbb{D}$ be the open unit disk in the complex plane $\mathbb{C}$. The
Hardy spaces $H^{p}(\mathbb{D})$ is defined as
$$H^{p}(\mathbb{D})= \{f\in \mathcal{H}(\mathbb{D}):\|f\|_{H_{p}}^{p}=\sup_{0<r<1}\int_{\partial\mathbb{D}}|f(r\zeta)|^{p}dm(\zeta),\; 0<p<\infty\}$$
where $m$ denotes normalised lebesgue measure on boundary $\partial \mathbb{D}$ of $\mathbb{D}$. The Bergman spaces $A_{\alpha}^{p}(\mathbb{D})$ is defined as
$$A_{\alpha}^{p}(\mathbb{D})= \{f\in \mathcal{H}(\mathbb{D}):\|f\|_{A_{\alpha}^{p}}^{p}=\int_{\mathbb{D}}|f(z)|^{p}dA_{\alpha}(z),\; 0<p<\infty,\; \alpha >-1\}$$
where $A$ denotes normalised Area measure on $\mathbb{D}$ and $dA_{\alpha}(z)=(\alpha+1)(1-|z|)^{2}d(A(z))$. With the above defined norms $H^{p}$ and $A_{\alpha}^{p}$ are Banach spaces, for $1\leq p < \infty$ and with the translation-invariant metric defined by $\rho(f,g)=\|f-g\|_{X}^{p}$, $X=H^{p}$ and $A^{p}_{\alpha}$ are complete metric spaces, for $0<p<1$, . Recall that a weight on $\mathbb{D}$ is a continuous function on $\mathbb{D}$ and let it be $\nu$. Then with the norm denoted by $\|\cdot\|_{\nu}^{\infty}$, the weighted Banach space $H_{\nu}^{\infty}$ is defined as
$$H_{\nu}^{\infty}=\{f\in \mathcal{H}(\mathbb{D}):\|f\|_{\nu}^{\infty}=\sup_{z\in \mathbb{D}}\nu(z)|f(z)|<\infty\}.$$
When $\nu(z)=(1-|z|^{2})^{\alpha}$, $\alpha>0$, $H_{(1-|z|^{2})^{\alpha}}^{\infty}=A^{-\alpha}$ known as growth space and when $\nu(z)=1$, $H_{1}^{\infty}=H^{\infty}$ is the space of bounded analytic function with the norm denoted by $\|\cdot\|_{\infty}$.

 For each $a\in \mathbb{D}$, the involutive automorphism that exchanges $a$ and $0$ is denoted by $\sigma_{a}$ and is defined as $\sigma_{a}(z)=\frac{a-z}{1-\bar{a}z}$. Let $d(a,z)=|\sigma_{a}(z)|=|\frac{a-z}{1-\bar{a}z}|$ be the pseudohyperbolic metric on $\mathbb{D}$. For pseudohyperbolic distance between self maps $\phi$ and $\psi$ of $\mathbb{D}$, we use simple notation $d(z)=d(\phi(z),\psi(z))$ .\\
  For Banach spaces $X$ and $Y$, an operator $S:X\longrightarrow Y$ is said to be bounded \rm (\rm resp. Compact \rm) \rm if it maps every bounded set in $X$ into another \rm (\rm resp. a relatively compact set\rm) \rm in $Y$. For quasi-banach space $X$, an operator $S:X\longrightarrow L^{q}(\mu)$, $0<q<\infty$ is said to be order bounded with positive measure $\mu$ on $\mathbb{D}$ or $\partial\mathbb{D}$ if for each $f$ in $B_{X}$ there is a nonnegative element $h$ in $L^q(\mu)$ such that $|Sf|\le h$. We know that $H^q=L^q(m)\cap \mathcal{H}(\mathbb{D})$ and $A_{\beta}^{q}=L^q(A_\beta)\cap \mathcal{H}(\mathbb{D})$, in view of this $S : X \rightarrow A_{\beta}^{q}$ \rm(\rm resp.,$H^{q}$\rm) \rm is order bounded with quasi-banach space $X$ if and only if  for each $f$ in $B_{X}$ there is a function $h$ in $L^{q}(A_{\beta})$ \rm (\rm resp.,$L^{q}(m)$\rm) \rm such that $|Sf|\le h$. Above notation $f\leq g$ reveals that $f(z)\leq g(z)\mu-$a.e.

 Let $\tau \in \mathcal{S}(\mathbb{D})$ and $\vec{u}=\{u_{k}\}_{k=0}^{n}$ belongs to $\mathcal{H}(\mathbb{D})$ . Then the weighted differentiation composition operator is defined as $D_{u_{k},\tau}^{k}(f)=u_{k}C_{\tau}D^{k}(f)=u_{k}f^{(k)}\circ\tau$, where  $D^{k}(f)=f^{(k)}$ is the kth-derivative of $f$ and $C_{\tau}(f)=f\circ\tau$ is the composition operator. When $k=0$, $D_{u_{0},\tau}^{0}(f)=u_{0}C_{\tau}(f)$ is simply weighted composition operator. Now we can define the finite sum of weighted differentiation composition operator as $$S^{n}_{\vec{u},\tau}(f)(z)=\displaystyle\sum_{k=0}^{n}D_{u_{k}
,\tau}^{k}(f)(z)=\sum_{k=0}^{n}u_{k}C_{\tau}D^{k}(f)(z)
=\sum_{k=0}^{n}u_{k}f^{(k)}(\tau)(z)$$
where $f\in\mathcal{H}(\mathbb{D})$.\\
Throughout this paper, being variation in the values of constants at each occurrence we will take notation of constants as $C$. The expression $A\lesssim B$ or $B\gtrsim A$ indicates that there exists a constant $C$ such that $A \le CB$ and the expression $A\asymp B$ indicates that both $A\lesssim B$ and $B\lesssim A$ holds.\\

\section{Preliminaries}
This section contains the collection of some lemmas to be used in our main results. For more details we refers \cite{SMX} and \cite{ST}.\\
Let us consider a common parameter $\gamma$ for the space $X$ defined as follows.
\begin{equation}\label{1a}
 \gamma=\Bigg\{
\begin{array}{cccc}
  0&,\text{ if } X=H^{\infty}\\
  \alpha&,\text{ if } X=A^{-\alpha}\\
  \frac{\alpha+2}{p}&,\text{ if } X=A^{p}_{\alpha} \\
  \frac{1}{p}&,\text{ if } X=H^{p}
\end{array}
\end{equation}
\begin{lem}\label{1l}\cite{SMX} Let $p_{n}(z)=z^{n}$. Then we have that for $n\longrightarrow \infty$
\begin{equation*}
  \|p_n\|_X\asymp\Bigg\{\begin{array}{cc}
                          n^\gamma &, \text{ if } X=H^\infty , A^{-\alpha} \\
                          n^{-(\gamma-\frac{1}{p})} &, \text{ if } X=H^{p}, A_{\alpha}^{p}
                        \end{array}
\end{equation*}
\end{lem}

\begin{lem}\label{2l}\cite{SMX}
  Let $X=H^{\infty}, A^{-\alpha}, A_{\alpha}^{p}, H^{p}$. Then there is a non-zero and positive real no. $C$ such that
  $$|f^{(k)}(z)|\leq \frac{C\|f\|_{X}}{(1-|z|^{2})^{k+\gamma}},$$
   for every $f\in X$, $0\leq k\leq n$ and $z\in \mathbb{D}$.
\end{lem}

\begin{lem}\label{3l}\
  Let $X=H^{\infty}, A^{-\alpha}, A_{\alpha}^{p}, H^{p}$ and $f_a(z)=\displaystyle\left(\frac{1-|a|^2}{(1-\bar{a}z)^{2}}\right)^{\gamma}$, for any $a,z\in \mathbb{D}$. Then $f_{a}\in X$ and $\|f_{a}\|_{X}\leq 1$. Also $\|f_{a}(\sigma_{a})^{k}\|_{X}\leq 1$ for  $a\in \mathbb{D}$ and $k\in \mathbb{N}$.
\end{lem}
\begin{pf}
  This lemma is trivially hold for $H^\infty$. For the remaining proof we refer the readers \cite{SMX}, \cite{CBM} and \cite{KZ}.
\end{pf}
\begin{lem}\label{4l}
  Let $X=H^{\infty}, A^{-\alpha}, A_{\alpha}^{p}, H^{p}$ and $Y=H^{\infty}_{\nu}$. Then $S^{n}_{\vec{u},\tau}: X \longrightarrow Y$ is compact if and only if $S^{n}_{\vec{u},\tau}: X \longrightarrow Y$ is bounded and $\lim_{n \rightarrow \infty}\|S^{n}_{\vec{u},\tau}(f_n)\|_{H^{\infty}_{\nu}}= 0$ for each bounded sequence $\{f_n\}_{n\in \mathbb{N}}$ in $X$ which is uniformly convergent to zero on compact subsets of $\mathbb{D}$.
\end{lem}
\section{Boundedness of $ S^{n}_{\vec{u},\tau}$}

This section is devoted to the characterization of the boundedness of the finite sum of weighted composition differentiation operator $S^{n}_{\vec{u},\tau}$ between Banach spaces of an analytic function.

\begin{theorem}\label{1t}\begin{enumerate}\item[{(A)}]Suppose that $X=H^{\infty}, A^{-\alpha}$. Then the following conditions are equivalent:
\begin{enumerate}
\item [{(a)}] $S^{n}_{\vec{u},\tau}: X \longrightarrow H_{\nu}^{\infty}$ is bounded.
\item [{(b)}]$\displaystyle\sum_{k=0}^{n}M_{k}=\displaystyle\sup_{z\in\mathbb{D}}\sum_{k=0}^{n}\frac{\nu(z)|u_{k}
    (z)|}{(1-|\tau(z)|^2)^{k+\gamma}}< \infty$.
\item[{(c)}]$M_{k}=\displaystyle\sup_{z\in\mathbb{D}}\frac{\nu(z)|u_{k}(z)|}{(1-|\tau(z)|^2)^{k+
    \gamma}}< \infty $, for $0\leq k\leq n$.
\item [{(d)}]$\displaystyle\sup_{a \in \mathbb{D}}\|S^{n}_{\vec{u},\tau}(f_{a}(\sigma_{a})^{k})\|_{H_{\nu}^{
    \infty}}< \infty$ for $0\leq k\le n$.
\item [{(e)}]$\displaystyle\sup_{n \in \mathbb{N}}n^{\gamma}\|S^{n}_{\vec{u},\tau}p_{n}\|_{H_{\nu}^{\infty}}< \infty$ where $p_{n}(z)=z^{n}$, $z\in \mathbb{D}$.
\end{enumerate}
\item [{(B)}] Conditions $(a)-(d)$ are equivalent for $X=A_{\alpha}^{p}, H^{p}$.
                     \end{enumerate}
\end{theorem}
\begin{proof}
Proof of the part (A):\\ $ (a)\Rightarrow (e)$ This implication trivially holds by lemma \ref{1l}.\\
$(e)\Rightarrow(d)$ Suppose that $\sup_{n \in \mathbb{N}}n^{\gamma}\|S^{n}_{\vec{u},\tau}p_{n}\|_{H_{\nu}^{\infty}}< \infty$ where $p_{n}(z)=z^{n}$, $z\in \mathbb{D}$.
We have to show that $\sup_{z\in\mathbb{D}}\|S^{n}_{\vec{u},\tau}(f_{a}(\sigma_{a})^{k})\|_{H_
{\nu}^{\infty}}< \infty$ for $0\leq k\le n$. If $\gamma\ge0$ and  $0\leq k\le n$, then by binomial series formula and Stirling's formula, we have that
\begin{equation}\label{eq}
\frac{1}{(1-\bar{a}z)^{2\gamma+k}}=\displaystyle\sum_{l=k}^{\infty}
\frac{\Gamma(l+2\gamma)}{(l-k)!\Gamma(2\gamma+k)}\bar{a}^{l-k}z^
{l-k}\asymp \displaystyle\sum_{l=k}^{\infty}\frac{\Gamma(l+\gamma)}{(l-k)!\Gamma(k+\gamma)}l^{\gamma}\bar{a}^{l-k}z^{l-k}
\end{equation}
and\begin{equation}\label{equ}
    \frac{1}{(1-|a|)^{\gamma+k}}=\displaystyle\sum_{l=k}^{\infty}\frac{\Gamma(l+\gamma)}{(l-k)!\Gamma(\gamma+k)}|a|^{l-k}
   \end{equation}
 For $k=0$, $\gamma\ge0$, by (\ref{eq}) and (\ref{equ}), we have that
\begin{eqnarray}\label{2a}
\big|S^{n}_{\vec{u},\tau}(f_{a})(z)\big|&= &\bigg|S^{n}_{\vec{u},\tau}\left(\frac{1-|a|^{2}}{(1-
\bar{a}z)^{2}}\right)^{\gamma}\bigg|\nonumber\\
&\asymp &(1-|a|^{2})^{\gamma}\bigg|S^{n}_{\vec{u},\tau}
\bigg(\sum_{l=0}^{\infty}\frac{\Gamma(l+\gamma)}{l!\Gamma(\gamma)}
l^{\gamma}\bar{a}^{l}z^l\bigg)\bigg|\nonumber\\
&\leq &(1-|a|^{2})^{\gamma}\sum_{l=0}^{\infty}\frac{\Gamma(l+\gamma)}{l!\Gamma(\gamma)}|a|^{l}l^
{\gamma}\big|S^{n}_{\vec{u},\tau}(p_{l})\big|\nonumber\\
&\leq & \sup_{n\in \mathbb{N}}n^{\gamma}\big|S^{n}_{\vec{u},\tau}(p_n)(z)\big|.
\end{eqnarray}
Since
\begin{eqnarray}\label{3a}
  (a-\sigma_a)^{k}& =&\sum_{j=0}^{k}\binom{k}{j}a^{j}(-\sigma_a)^{k-j}\notag\\
  &=&(-\sigma_a)^{k}+\sum_{j=0}^{k-1}\binom{k}{j}a^{j}
  (-\sigma_a)^{k-j}.
\end{eqnarray}
Therefore for $1\leq k\leq n$, in order to prove $$\sup_{a\in\mathbb{D}}\big\|S^{n}_{\vec{u},\tau}\left(f_{a}\sigma_a^{k}
\right)\big\|\leq\sup_{n\in\mathbb{N}}n^{\gamma}\big\|
S^{n}_{\vec{u},\tau}(p_{n})
\big\| $$
we only need to prove $$\sup_{a\in\mathbb{D}}\big|S^{n}_{\vec{u},\tau}\left(f_{a}(a-\sigma_a)^{k}\right)(z)
\big|\leq\sup_{n\in\mathbb{N}}n^{\gamma}\big
|S^{n}_{\vec{u},\tau}(p_{n})(z)\big|.$$
If $1\leq k\leq n$, then by binomial series formula, we have that \begin{equation}\label{equa}
\frac{1}{\left(1-|a|\right)^{k}}=\displaystyle\sum_{l=k}^{\infty}\binom{l-1}
{k-1}|a|^{l-k}
\end{equation}
Thus, if $\gamma=0$, then by (\ref{equa}), we have that
\begin{align}\label{4a}
  \big|S^{n}_{\vec{u},\tau}((a-\sigma_{a})^{k})(z)\big|
  &=\bigg|S^{n}_{\vec{u},\tau}\bigg(\frac{(1-|a|^{2})^{k}z^k}{(1-
\bar{a}z)^{k}}\bigg)\bigg|\notag\\
  &=\bigg|S^{n}_{\vec{u},\tau}\bigg(\big(1-|a|^{2}\big)^{k}\sum_{l=k}^{\infty}
  \binom{l-1}{k-1}\bar{a}^{l-k}z^{l}\bigg)\bigg|\notag\\
  &\leq\big(1-|a|\big)^{k}\sum_{l=k}^{\infty}\binom{l-1}{k-1}|a|^
  {l-k}\big|S^{n}_{\vec{u},\tau}(p_{l})(z)\big|\notag\\
  &\lesssim \sup_{n\in\mathbb{N}}\big|S^{n}_{\vec{u},\tau}(p_{n})(z)\big|.
\end{align}
If $\gamma>0$, then by (\ref{eq}) and(\ref{equ}), we have that
\begin{align}\label{5a}
  \big|S^{n}_{\vec{u},\tau}\left(f_{a}(a-\sigma_a)^{k}\right)(z)\big|
  &=\bigg|S^{n}_{\vec{u},\tau}\left(\frac{(1-|a|^{2})^{\gamma+k}z^k}{(1-
\bar{a}z)^{2\gamma+k}}\right)\bigg|\notag\\
  &\asymp\bigg|S^{n}_{\vec{u},\tau}\bigg((1-|a|^{2})^{\gamma+k}\sum_{l=k}^{\infty}
  \frac{\Gamma(l+\gamma)}{(l-k)!\Gamma(\gamma+k)}l
  ^{\gamma}\bar{a}
  ^{l-k}z^{l}\bigg)\bigg|\notag\\
  &\leq(1-|a|)^{\gamma+k}\sum_{l=k}^{\infty}\frac{
  \Gamma(l+\gamma)}{(l-k)!\Gamma(\gamma+k)}|a|^{l-k}l
  ^{\gamma}\big|S^{n}_{\vec{u},\tau}(p_{l})(z)\big|\notag\\
  &\lesssim \sup_{n\in\mathbb{D}}n^{\gamma}\big|S^{n}_{\vec{u},\tau}(p_{n})(z)\big|.
  \end{align}
  Inequalities (\ref{2a}), (\ref{4a}) and (\ref{5a}), reveals that $$ \sup_{z\in\mathbb{D}}\big|S^{n}_{\vec{u},\tau}(f_{a}(\sigma_{a})^{k})\big|_{H_
{\nu}^{\infty}}\leq \sup_{n \in \mathbb{N}}\big|S^{n}_{\vec{u},\tau}(p_{n})\big|_{H_{\nu}^{\infty}}<\infty$$
This proves the implication $(a)\Rightarrow (d)$.\\

$(d)\Rightarrow(c)$ Suppose that $\sup_{a\in \mathbb{D}}\big\|S^{n}_{\vec{u},\tau}(f_{a}\sigma_{a}^{k})\big\|< \infty$, $0\leq k\leq n$. We have to show that, for each $0\leq k\leq n$,
$M_k=\displaystyle\sup_{z\in\mathbb{D}}\frac{\nu(z)|u_{k}(z)|}{\left(1-|\tau(z)|^{2}\right)^{k+\gamma}}< \infty$
For $k=n$, we consider $g_{n}(z)=\big(f_{\tau(z)}\sigma_{\tau(z)}^{n}\big)(z)$, for all $z\in \mathbb{D}$ such that $g_{n}^{(1)}
(\tau(z))=g_{n}^{(2)}
(\tau(z))=\cdots=g_{n}^{(n)}
(\tau(z))=0$ and $g_{n}^{(n)}
(\tau(z))=\displaystyle\frac{n!}{\left(1-|\tau(z)|^{2}\right)
^{n+\gamma}}$. Thus
$$\big|S_{\vec{u},\tau}(g_{n})(z)\big|=\bigg|\sum_{i=0}^{n}u_{i}(z)g_{n}^{(i)}
(\tau(z))\bigg|=\frac{n!|u_{n}(z)|}{\big(1-|\tau(z)|^{2}\big)
^{n+\gamma}}$$
This implies that,
\begin{equation}\label{6a}
M_{n}\lesssim \sup_{a\in\mathbb{D}}\|S^{n}_{\vec{u},\tau}(f_{a}(\sigma_{a})^{n})\|_{H_{\nu}
^{\infty}}< \infty
\end{equation}
Next, for $k=n-1$, we consider $g_{n-1}(z)=\big(f_{\tau(z)}\sigma_{\tau(z)}^{n-1}\big)
(z)$, $z\in \mathbb{D}$ such that $g_{n-1}^{(1)}
(\tau(z))=g_{n-1}^{(2)}
(\tau(z))=\cdots=g_{n-1}^{(n-2)}
(\tau(z))=0$ and $g_{n-1}^{(n-1)}
(\tau(z))=\displaystyle\frac{(n-1)!}{\left(1-|\tau(z)|^{2}\right)
^{n-1+\gamma}}$. By lemma[\ref{3l}], $\|g_{n-1}\|_{X}\leq 1$. Thus, by lemma[\ref{2l}]
\begin{align}\label{7a}
  \big|S^{n}_{\vec{u},\tau}(g_{n-1})(z)\big|&= \bigg|\sum_{i=0}^{n}u_{i}(z)g_{n-1}^{(i)}(\tau(z))\bigg| \notag\\
  &\geq \big|u_{n-1}(z)g_{n-1}^{(n-1)}(\tau(z))\big|-\big|u_{n}(z)
  g_{n-1}^{(n)}(\tau(z))\big|\notag\\
  &\geq \frac{(n-1)!|u_{n-1}(z)|}{\left(1-|\tau(z)|^{2}\right)^
  {n-1+\gamma}}-\frac{C|u_{n}(z)|\|g_{n-1}\|_X}
  {\left(1-|\tau(z)|^{2}\right)^{n+\gamma}}.
\end{align}
By (\ref{6a}) and (\ref{7a}), we have that
\begin{align}\label{8a}
  M_{n-1}&=\sup_{z\in\mathbb{D}}\frac{\nu(z)|u_{n-1}(z)|}{\left(1-|\tau(z)|^{2}\right)^
  {n-1+\gamma}}\notag\\
  &\lesssim C\sup_{z\in\mathbb{D}}\frac{\nu(z)|u_{n}(z)|\|g_{n-1}\|_X}
  {\left(1-|\tau(z)|^{2}\right)^{n+\gamma}}+\sup_{z\in\mathbb{D}}\nu(z)\big|S_{\vec{u},\tau}(g_{n-1})(z)\big|\notag\\ &\leq C M_n+ \sup_{a\in\mathbb{D}}\big\|S^{n}_{\vec{u},\tau}(f_{a}(\sigma_{a})^{n-1})\big\|_
  {H_{\nu}^{\infty}}\notag\\
  &< \infty.
\end{align}
Using  lemma \ref{2l}, (\ref{6a}) and (\ref{8a}) , similarly we can proceed for $k=n-2$, and we find that
 \begin{equation}\label{9a}M_{n-2}< \infty
 \end{equation}
 Further, for $j>k$ assume that
 \begin{equation}\label{10a}
   M_{j}<\infty
 \end{equation}
For $k$, consider $g_{k}(z)=\big(f_{\tau(z)}\sigma_{\tau(z)}^{k}\big)(z)$ such that $g_{k}^{(1)}
(\tau(z))=g_{k}^{(2)}
(\tau(z))=\cdots=g_{k}^{(k-1)}
(\tau(z))=0$ and $g_{k}^{(k)}
(\tau(z))=\displaystyle\frac{k!}{\left(1-|\tau(z)|^{2}\right)
^{k+\gamma}}$. By lemma \ref{3l}, $\|g_{k}\|_{X}\leq 1$. Thus, by lemma \ref{2l}, we have that
\begin{align*}
 \big|S^{n}_{\vec{u},\tau}(g_{k})(z)\big|&= \bigg|\sum_{i=0}^{n}u_{i}(z)g_{k}^{(i)}(\tau(z))\bigg| \notag\\
  &\geq \big|u_{k}(z)g_{k}^{(k)}(\tau(z))\big|-\sum_{i=k+1}^{n}\big|u_{i}(z)(z)g_{k}^
  {(i)}(\tau(z))\big|\notag\\
  &\geq \frac{k!|u_{k}(z)|}{\left(1-|\tau(z)|^{2}\right)^{k+
  \gamma}}-\sum_{i=k+1}^{n}\frac{C|u_{i}(z)|\|g_{k}\|_X}
  {\left(1-|\tau(z)|^{2}\right)^{i+\gamma}}.
\end{align*}
Using (\ref{6a}), (\ref{8a}), (\ref{9a}) and (\ref{10a}), we find that
\begin{align*}
  M_{k}&=\sup_{z\in\mathbb{D}}\frac{\nu(z)|u_{k}(z)|}{\left(1-|\tau(z)|^{2}\right)^{k+
  \gamma}}\notag\\
  &\lesssim \sum_{i=k+1}^{n}C\sup_{z\in\mathbb{D}}\frac{\nu(z)|u_{i}(z)|\|g_{k}\|_X}
  {\left(1-|\tau(z)|^{2}\right)^{i+\gamma}} +\sup_{z\in\mathbb{D}}\nu(z)\big|S^{n}_{\vec{u},\tau}(g_{k})(z)\big|\notag\\
  &\leq\sum_{i=k+1}^{n} C M_i+ \sup_{a\in\mathbb{D}}\big\|S^{n}_{\vec{u},\tau}(f_{a}(\sigma_{a})^{k})\big\|_{H_
  {\nu}^{\infty}}\notag\\
  &< \infty.
\end{align*}
This proves the implication $(d)\Rightarrow(c)$.\\
$(c)\Rightarrow (b)$ This implication is trivial one.\\
$(b)\Rightarrow (a)$ Assume that $\sum_{k=0}^{n}M_{k}=\displaystyle\sup_{z\in\mathbb{D}}\sum_{k=0}^{n}\frac{\nu(z)|u_{k}
    (z)|}{(1-|\tau(z)|^2)^{k+\gamma}}=M$. By lemma\ref{2l}, we have that
\begin{align*}
 \nu(z)\big|S^{n}_{\vec{u},\tau}(f)(z)\big|&= \nu(z)\bigg|\sum_{k=0}^{n}u_{k}(z)f^{(k)}(\tau(z))\bigg| \notag\\
  &\leq \sum_{k=0}^{n}\nu(z)|u_{k}(z)||f^{(k)}(\tau(z))| \notag\\
   &\leq\sum_{k=0}^{n}\frac{C\nu(z)|u_{k}(z)|\|f\|_{X}}{\left(1-
   |\tau(z)|^{2}\right)^{k+\gamma}}
\end{align*}
Thus, we have that $\big\|S^{n}_{\vec{u},\tau}f\big\|_{H_{\nu}^{\infty}(z)} \lesssim M\big\|f\big\|_{X}$ which implies that $S^{n}_{\vec{u},\tau}:X \longrightarrow H_{\nu}^{\infty}$ is bounded.\\
Proof of the part(B):\\
The proofs of the implications $(d)\Rightarrow (c)\Rightarrow (b)\Rightarrow (a)$ are the same as the proof of part (A).The proof of the implication $(a)\Rightarrow (d)$ follows from lemma\ref{3l}.
This completes the proof of the theorem\ref{1t} .\end{proof}

\section{Compactness of $ S^{n}_{\vec{u},\tau}$}\label{sec:prelimi}
This section is devoted to characterizing the  Compactness of the finite sum of the weighted composition differentiation operator.

\begin{theorem}\label{2t} \begin{enumerate}
                       \item [{(A)}]Suppose that $X=H^{\infty}, A^{-\alpha}$. Then the following conditions are equivalent:
                           \begin{enumerate}
\item [{(i)}] $S^{n}_{\vec{u},\tau}: X \longrightarrow  H_{\nu}^{\infty}$ is compact.
\item[{(ii)}]$\sum_{k=0}^{n}G_{k}=\displaystyle\sum_{k=0}^{n}\lim_{|\tau(z)|\rightarrow 1}\frac{\nu(z)|u_{k}(z)|}
    {(1-|\tau(z)|^2)^{k+\gamma}}=0$.
\item [{(iii)}]$G_{k}=\displaystyle\lim_{|\tau(z)|\rightarrow 1}\frac{\nu(z)|u_{k}(z)|}{(1-|\tau(z)|^2)^{k+\gamma}}=0$, for $0\leq k\leq n$.
\item [{(iv)}]$\lim_{|a|\rightarrow 1}\|S^{n}_{\vec{u},\tau}(f_{a}(\sigma_{a})^{k})\|_{H_{\nu}^{\infty}}=0$ for $0\leq k\le n$.
\item [{(v)}]$\lim_{n\rightarrow \infty}n^{\gamma}\|S^{n}_{\vec{u},\tau}(p_{n})\|_{H_{\nu}^{\infty}}=0$ where $p_{n}(z)=z^{n}$, $z\in \mathbb{D}$.
\end{enumerate}

\item [{(B)}] Conditions $(i)-(iv)$ are equivalent for $X=A_{\alpha}^{p}, H^{p}$.
                     \end{enumerate}
\end{theorem}
\begin{proof}$(i)\Rightarrow (v)$ This implication is trivially holds by lemma \ref{2l}.\\
$(v)\Rightarrow (iv)$ Suppose that $\lim_{n\rightarrow \infty}n^{\gamma}\|S^{n}_{\vec{u},\tau}p_{n}\|_{H_{\nu}^{\infty}}=0$ where $p_{n}(z)=z^{n}$, $z\in \mathbb{D}$. we have to show that $\lim_{|a|\rightarrow 1}\|S^{n}_{\vec{u},\tau}(f_{a}(\sigma_{a})^{k})\|_{H_{\nu}^{\infty}}=0$ for $0\leq k\le n$.\\
For $k=0$, by using (\ref{eq}), we get that
\begin{eqnarray}\label{4b}
  \big\|S^{n}_{\vec{u},\tau}(f_{a})(z)\big\|_
  {H_{\nu}^{\infty}}&=&\bigg\|S^{n}_{\vec{u},\tau}\bigg(\frac{1-|a|^{2}}{(1-
\bar{a}z)^{2}}\bigg)^{\gamma}\bigg\|_
  {H_{\nu}^{\infty}}\nonumber\\
 &\asymp &(1-|a|^{2})^{\gamma}\bigg\|S^{n}_{\vec{u},\tau}\bigg(\sum_{l=0}^{\infty}
 \frac{\Gamma(l+\gamma)}{l!\Gamma(\gamma)}l^{\gamma}
 \bar{a}^{l}z^l\bigg)\bigg\|_{H_{\nu}^{\infty}}\nonumber\\
 &\le &(1-|a|^{2})^{\gamma}\sum_{l=0}^{\infty}\frac{\Gamma
 (l+\gamma)}{l!\Gamma(\gamma)}l^{\gamma}|a|^{l}\big|S^{n}_{\vec{u},\tau}(p_{l})\big|_
 {H_{\nu}^{\infty}}\nonumber\\
 &= &(1-|a|^{2})^{\gamma}\sum_{l=0}^{N}\frac{\Gamma
 (l+\gamma)}{l!\Gamma(\gamma)}l^{\gamma}|a|^{l}\big\|S^{n}_{\vec{u},\tau}(p_{l})\big\|_{H_{\nu}^{\infty}}\notag\\
 &\quad &+(1-|a|^{2})^{\gamma}\sum_{l=N+1}^{\infty}\frac{\Gamma
 (l+\gamma)}{l!\Gamma(\gamma)}l^{\gamma}|a|^{l}l^{\gamma}\big|
 S^{n}_{\vec{u},\tau}(p_{l})\big|_{H_{\nu}^{\infty}}\nonumber\\
 &\leq &(1-|a|^{2})^{\gamma}\sum_{l=0}^{N}\frac{\Gamma(l+2
 \gamma)}{l!\Gamma(2l)}l^{\gamma}|a|^{l}\big|S^{n}_{\vec{u},\tau}(p_{l})\big|_{H_{\nu}^
 {\infty}}\nonumber\\
 &\quad &+(1-|a|^{2})^{\gamma}\bigg(\sum_{l=0}^{\infty}\frac{\Gamma
 (l+\gamma)}{l!\Gamma(\gamma)}|a|^{l}\bigg)\sup_{n>N}n^
 {\gamma}\big|S^{n}_{\vec{u},\tau}(p_{n})\big|_{H_{\nu}^{\infty}}.
\end{eqnarray}
Since $\gamma>0$ and $N$ is arbitrary, using (\ref{equ}),  we get that
$$\lim_{|a|\rightarrow 1}\big\|S^{n}_{\vec{u},\tau}(f_{a})\big\|_{H_{\nu}^{\infty}}\leq \sup_{n>N}n^{\gamma}\big\|S^{n}_{\vec{u},\tau}(p_{n})\big\|_{H_{\nu}^{\infty}}$$
implies
$$\lim_{|a|\rightarrow 1}\|S^{n}_{\vec{u},\tau}(f_{a})\|_{H_{\nu}^{\infty}}\leq\lim_{n\rightarrow \infty}n^{\gamma}\big\|S^{n}_{\vec{u},\tau}(p_{n})\big\|_{H_{\nu}^{\infty}}=0.$$
Next, assume that $\lim_{|a|\rightarrow 1}\big\|S^{n}_{\vec{u},\tau}(f_{a}(\sigma_{a})^{j})\big\|_{H_{\nu}^{\infty}}=0$
holds for , $j<k$. So, in order to prove $$\lim_{|a|\rightarrow 1}\big\|S^{n}_{\vec{u},\tau}(f_{a}(\sigma_{a})^{k})(z)\big\|_{H_{\nu}^{\infty}}=0,$$ we need only to prove $$\lim_{|a|\rightarrow 1}\big|S^{n}_{\vec{u},\tau}(f_{a}(a-\sigma_{a})^{k})\big|=0.$$\\
For arbitrary $N\in \mathbb{N}$ and $N\geq k$, using(\ref{eq}) we have that
\begin{eqnarray}\label{5b}
  \big\|S^{n}_{\vec{u},\tau}(f_{a}(a-\sigma_{a})^{k})\big\|_
  {H_{\nu}^{\infty}}&= &\bigg\|S^{n}_{\vec{u},\tau}\left(\frac{(1-|a|^{2})^
  {\gamma+k}z^k}{(1-\bar{a}z)^{2\gamma+k}}\right)\bigg\|_
  {H_{\nu}^{\infty}}\nonumber\\
 &\asymp &\bigg\|S^{n}_{\vec{u},\tau}\bigg((1-|a|^{2})^{\gamma+k}\sum_{l=k}^{\infty}
  \frac{\Gamma(l+\gamma)}{(l-k)!\Gamma(\gamma+k)}l
  ^{\gamma}\bar{a}^{l-k}z^{l}\bigg)\bigg\|_{H_{\nu}^
  {\infty}}\nonumber\\
 &\le &(1-|a|^{2})^{\gamma+k}\sum_{l=k}^{\infty}\frac{\Gamma
 (l+\gamma)}{(l-k)!\Gamma(\gamma+k)}l^{\gamma}|a|^{l-k}
 \|S^{n}_{\vec{u},\tau}(p_{l})\|_{H_{\nu}^{\infty}}\nonumber\\
 &\leq &(1-|a|^{2})^{\gamma+k}\sum_{l=k}^{N}\frac{\Gamma(l+
  \gamma)}{(l-k)!\Gamma(\gamma+k)}l^{\gamma}|a|^{l-k}
  \|S^{n}_{\vec{u},\tau}(p_{l})\|_{H_{\nu}^{\infty}}\nonumber\\
 &\quad &+(1-|a|^{2})^{\gamma+k}\bigg(\sum_{l=N+1}^{\infty}\frac{
 \Gamma(l+\gamma)}{(l-k)!\Gamma(\gamma+k)}l^{\gamma}|a|^{l-k
 }\bigg)\sup_{n>N}n^{\gamma}\|S^{n}_{\vec{u},\tau}(p_{n})\|_{H_{\nu}^{\infty}}
 \end{eqnarray}
Since $\gamma>0$ and $N$ is arbitrary, by (\ref{equ})  we have that
$$\lim_{|a|\rightarrow 1}\big\|S^{n}_{\vec{u},\tau}(f_{a}(a-\sigma_{a})^{k})\big\|_{H_{\nu}^{\infty}}\leq\sup_{n>N}\big\|S_{\vec{u},\tau}(p_{n})\big\|_{H_{\nu}^{\infty}}$$
implies
$$\lim_{|a|\rightarrow 1}\big\|S^{n}_{\vec{u},\tau}(f_{a}(a-\sigma_{a})^{k})\big\|_{H_{\nu}^{\infty}}\leq\lim_
{n\rightarrow \infty}\big\|S^{n}_{\vec{u},\tau}(p_{n})\big\|_{H_{\nu}^{\infty}}=0.$$

$(iv)\Rightarrow (iii)$ Suppose that $\lim_{|a|\rightarrow 1}\big\|S^{n}_{\vec{u},\tau}(f_{a}(a-\sigma_{a})^{k})\big\|_{H_{\nu}^{\infty}}=0$. We have to show that $M_{k}=\displaystyle\lim_{|\tau(z)|\rightarrow 1}\frac{|u_{k}(z)|}{(1-|\tau(z)|^2)^{k+\gamma}}=0$, for $0\leq k\leq n$. For $k=n$, for all $z\in \mathbb{D},$ consider $g_{n}(z)=\big(f_{\tau(z)}\sigma_{\tau(z)}^{n}\big)(z)$ such that $g_{n}^{(1)}
(\tau(z))=g_{n}^{(2)}
(\tau(z))=\cdots=g_{n}^{(n-1)}
(\tau(z))=0$ and $g_{n}^{(n)}
(\tau(z))=\displaystyle\frac{n!}{\left(1-|\tau(z)|^{2}\right)
^{n+\gamma}}$. Thus
$$\big|S^{n}_{\vec{u},\tau}(g_{n})(z)\big|=\bigg|\sum_{i=0}^{n}u_{i}(z)g_{n}^{(i)}
(\tau(z))\bigg|=\frac{n!|u_{n}(z)|}{\left(1-|\tau(z)|^{2}\right)
^{n+\gamma}}$$
This implies that,
\begin{equation}\label{6b}
  G_n=\lim_{|\tau(z)|\rightarrow 1}\frac{|u_{n}(z)|}{(1-|\tau(z)|^2)^{n+\gamma}} \lesssim\lim_{|\tau(z)|\rightarrow 1}\big\|S^{n}_{\vec{u},\tau}\left(f_{\tau(z)}\sigma_{\tau(z)}^{n}\right)\big\|_{H_{\nu}^{\infty}}=0 .
\end{equation}
Next, for $k=n-1$, we consider $g_{n-1}(z)=\big(f_{\tau(z)}\sigma_{\tau(z)}^{n-1}\big)
(z)$, $z\in \mathbb{D}$ such that $g_{n-1}^{(1)}
(\tau(z))=g_{n-1}^{(2)}
(\tau(z))=\cdots=g_{n-1}^{(n-2)}
(\tau(z))=0$ and $g_{n-1}^{(n-1)}
(\tau(z))=\displaystyle\frac{(n-1)!}{\left(1-|\tau(z)|^{2}\right)
^{n-1+\gamma}}$. By lemma \ref{3l}, $\|g_{n-1}\|_{X}\leq 1$. Thus by lemma \ref{2l}, we get that
\begin{eqnarray*}
   \big|S^{n}_{\vec{u},\tau}(g_{n-1})(z)\big|&\geq &\big|u_
   {n-1}(z)g_{n-1}^{(n-1)}(\tau(z))\big|-\big|
   u_{n}(z)g_{n-1}^{(n)}(\tau(z))\big|\nonumber\\
  &\geq & \frac{(n-1)!|u_{n-1}(z)|}{\left(1-|\tau(z)|^{2}\right)^{n-1+
  \gamma}}-\frac{C|u_{n}(z)|\|g_{n-1}\|_X}{\left(1-|\tau(z)|^{2}\right)^{n+\gamma}}.
\end{eqnarray*}
This implies that
\begin{eqnarray}\label{7b}
  G_{n-1}&= & \lim_{|\tau(z)|\rightarrow 1}\frac{\nu(z)|u_{n-1}(z)|}{(1-|\tau
   (z)|^2)^{n-1+\gamma}}\nonumber\\
   &\lesssim & C\lim_{|\tau(z)|\rightarrow 1} \frac{\nu(z)|u_{n}(z)|\|g_{n-1}\|_X}
  {\left(1-|\tau(z)|^{2}\right)^{n+\gamma}}+\lim_{|\tau(z)|\rightarrow 1}\nu(z)\big|S^{n}_{\vec{u},\tau}(g_{n-1})(z)\big|\nonumber\\
   &\leq & CG_n+\lim_{|\tau(z)|\rightarrow1}\big\|S^{n}_{\vec{u},\tau}(f_{
   \tau(z)}\sigma_{\tau(z)}^{n-1})\big\|_{H_{\nu}^{\infty}}
   \nonumber\\&= &0
\end{eqnarray}
Using (\ref{6b}) and (\ref{7b}), we can proceed for $k=n-2$ and we find that
\begin{equation}
 G_{n-2}= \lim_{|\tau(z)|\rightarrow 1}\frac{\nu(z)|u_{n-2}(z)|}{(1-|\tau(z)|^2)^{n-2+\gamma}}=0 \nonumber
\end{equation}
Further, assume that
\begin{equation}
  G_j=\lim_{|\tau(z)|\rightarrow 1}\frac{\nu(z)|u_{j}(z)|}{(1-|\tau(z)|^2)^{j+\gamma}}=0 \nonumber
\end{equation}
holds for $j>k$. For the proof completion of this implication, we prove the condition for $k$. Consider $g_{k}(z)=\big(f_{\tau(z)}\sigma_{\tau(z)}^{j}\big)(z)$ such that $g_{k}^{(1)}
(\tau(z))=g_{k}^{(2)}
(\tau(z))=\cdots=g_{k}^{(k-1)}
(\tau(z))=0$ and $g_{k}^{(k)}
(\tau(z))=\displaystyle\frac{k!}{\left(1-|\tau(z)|^{2}\right)
^{k+\gamma}}$. By lemma[\ref{3l}], $\|g_{k}\|_{X}\leq 1$. Thus by lemma(\ref{2l}), we get that
\begin{eqnarray*}
   \big|S^{n}_{\vec{u},\tau}(g_{k})(z)\big|&\geq &\big|
   u_{k}(z)g_{k}^{(k)}(\tau(z))\big|-\sum_{i=k+1}^{n}\big|
   u_{i}(z)g_{k}^{(i)}(\tau(z))\big|\nonumber\\
  &\geq & \frac{k!|u_{k}(z)|}{\left(1-|\tau(z)|^{2}\right)^{k+\gamma}}
  -\sum_{i=k+1}^{n}\frac{C|u_{i}(z)\big\|g_{k}\big\|_X}
  {\left(1-|\tau(z)|^{2}\right)^{i+\gamma}} .
\end{eqnarray*}
This implies that
\begin{eqnarray*}
G_k&= & \lim_{|\tau(z)|\rightarrow 1}\frac{\nu(z)|u_{k}(z)|}{(1-|\tau(z)|^2)^{k+\gamma}}\nonumber\\ &\lesssim & \sum_{i=k+1}^{n}\lim_{|\tau(z)|\rightarrow 1}\frac{C\nu(z)|u_{i}(z)|\|g_{k}\|_X}
  {\left(1-|\tau(z)|^{2}\right)^{i+\gamma}} +\lim_{|\tau(z)|\rightarrow 1}\nu(z)\big|S^{n}_{\vec{u},\tau}(g_{k})(z)\big|\nonumber\\
  &\lesssim &\sum_{i=k+1}^{n}CG_k+\lim_{|\tau(z)|\rightarrow 1}\|S^{n}_{\vec{u},\tau}(g_{k})\|_{H_{\nu}^{\infty}}\nonumber\\
&= &0
\end{eqnarray*}
$(iii)\Rightarrow (ii)$ This implication is trivial one.\\
$(ii)\Rightarrow (i)$ Suppose that the conditions holds. Let $\epsilon>0$, $r$ lies in $(0,1)$ and $|\tau(z)|>r$, by assumption we can write for $1\leq k\leq n$ \begin{equation}\frac{\nu(z)|u_{k}(z)|}{(1-|\tau(z)|^2)^
{k+\gamma}}< \epsilon\nonumber .\end{equation}
For $z\in \mathbb{D}$ with $|\tau(z)|\le r$, we have that
\begin{equation}
  \nu(z)\big|S^{n}_{\vec{u},\tau}(f_{n})(z)\big|\leq \sum_{k=0}^{n}\nu(z)|u_{k}(z)|\big|f_{n}^{k}(\tau(z))\big|\nonumber
\end{equation}
converges uniformly to zero by lemma \ref{4l} .\\
Next for $|\tau(z)|>r$, by lemma \ref{2l}, we have that
\begin{eqnarray*}
  \nu(z)\big|S^{n}_{\vec{u},\tau}(f_{n})(z)\big| &\leq & \sum_{k=0}^{n}\nu(z)|u_{k}(z)|\big|f_{n}^{k}(\tau(z))\big|\nonumber\\
  &\leq & \sum_{k=0}^{n}\frac{C\nu(z)|u_{k}(z)|\|f_{n}\|_{X}}{\left(1-|\tau(z)|^{2}\right)^{k+{\gamma}}}\nonumber\\
  \Rightarrow \limsup_{n\rightarrow \infty}\sup_{|\tau(z)|>r}\big\|S^{n}_{\vec{u},\tau}(f_{n}(z))\big\|_{H_{\nu}^{\infty}}&\lesssim \epsilon.
\end{eqnarray*}
Since $\epsilon$ is arbitrary, we get that
$$\big\|S^{n}_{\vec{u},\tau}(f_{n})(z)\big\|_{H_{\nu}^{\infty}}\longrightarrow 0$$
uniformly for each $z\in \mathbb{D}$. This proof holds for $X=A^{-\alpha}$, $A_{\alpha}^{p}$,$H^{p}$.

Proof of the part(B):
The proof of the implications (iv)$\Rightarrow $ (iii)$\Rightarrow$ (ii)$\Rightarrow$ (i)is the same as the proof of part (A). The proof of the implication (i)$\Rightarrow$ (iv) follows from lemma \ref{3l}.
This completes the proof of the Theorem \ref{2t}.
\end{proof}

\section{Order boundedness of $ S^{n}_{\vec{u},\tau}$}\label{sec:prelimi}
This section is devoted to characterizing the order boundedness of the finite sum of the weighted composition differentiation operator.
\begin{theorem}\label{3t}\begin{enumerate}
                       \item [{(A)}] Suppose that $X=H^{\infty}, A^{-\alpha}$. Then the following conditions are equivalent: \begin{enumerate}
\item [{(i)}] $S^{n}_{\vec{u},\tau}: X \longrightarrow L^q(\mu)$ is order bounded.
\item [{(ii)}]$\displaystyle\sum_{k=0}^{n}Q_{k}=\sup_{z\in\mathbb{D}}\sum_{k=0}^{n}\frac{|u_{k}
    (z)|}{(1-|\tau(z)|^2)^{k+\gamma}}\in L^q(\mu)$.
\item[{(iii)}]$Q_{k}=\displaystyle\sup_{z\in\mathbb{D}}\frac{|u_{k}(z)|}{(1-|\tau(z)|^2)^{k+
    \gamma}}\in L^q(\mu) $, for $0\leq k\leq n$.
\item [{(iv)}]$\displaystyle\sup_{a \in \mathbb{D}}|S^{n}_{\vec{u},\tau}(f_{a}(\sigma_{a})^{k})|\in L^q(\mu)$ for $0\leq k\le n$.
\item [{(v)}]$\displaystyle\sup_{n \in \mathbb{N}}n^{\gamma}|S^{n}_{\vec{u},\tau}(p_{n})| \in L^q(\mu)$ where $p_{n}(z)=z^{n}$, $z\in \mathbb{D}$.
\end{enumerate}

                       \item [{(B)}] Conditions (i)-(iv) are equivalent for $X=A_{\alpha}^{p}, H^{p}$.
                     \end{enumerate}
\end{theorem}
\begin{proof}
  Proof of the part$(A)$: \\
  (i)$\Rightarrow $ (v) Assume that $S^{n}_{\vec{u},\tau}: X \longrightarrow H_{\nu}^{\infty}$ is order bounded. By definition, there is an $g\in L^q(\mu)$ such that $|S^{n}_{\vec{u},\tau}(f)|\leq g$ for every $f\in B_{X}$. By lemma \ref{1l}, we have that for $n\rightarrow \infty$
  \begin{equation}
    \sup_{n \in \mathbb{N}}n^{\gamma}|S^{n}_{\vec{u},\tau}(p_{n})|\asymp \sup_{n \in \mathbb{N}}\bigg|\frac{S^{n}_{\vec{u},\tau}(p_n)}{\|p_n\|}\bigg|\leq g \nonumber
  \end{equation}
  which implies that $\sup_{n \in \mathbb{N}}n^{\gamma}|S^{n}_{\vec{u},\tau}(p_{n})|\in L^q(\mu)$.

  (v)$\Rightarrow $ (iv) To complete the proof of this implication,  we need only to prove
  \begin{equation}
   \sup_{a \in \mathbb{D}}|S^{n}_{\vec{u},\tau}(f_{a}(\sigma_{a})^{k})|\leq \sup_{n \in \mathbb{N}}n^{\gamma}|S^{n}_{\vec{u},\tau}(p_{n})|\;\;\;\;\; \textit{for all } 0\leq k\leq n \nonumber
  \end{equation}
  which is already proved in the proof of the theorem \ref{1t}.

(iv)$\Rightarrow $ (iii) Assume that $\sup_{a \in \mathbb{D}}|S^{n}_{\vec{u},\tau}(f_{a}(\sigma_{a})^{k})|\in L^q(\mu)$, for each $0\leq k\leq n$.
We have to show that, for each $0\leq k\leq n$,
$Q_k=\displaystyle\frac{|u_{k}(z)|}{\left(1-|\tau(z)|^{2}\right)^{k+\gamma}}\in L^q(\mu)$
For $k=n$, we consider $g_{n}(z)=\big(f_{\tau(z)}\sigma_{\tau(z)}^{n}\big)(z)$, for all $z\in \mathbb{D}$ such that $g_{n}^{(1)}
(\tau(z))=g_{n}^{(2)}
(\tau(z))=\cdots=g_{n}^{(n)}
(\tau(z))=0$ and $g_{n}^{(n)}
(\tau(z))=\displaystyle\frac{n!}{\left(1-|\tau(z)|^{2}\right)
^{n+\gamma}}$. By lemma \ref{3l}, $\|g_{n-1}\|_{X}\leq 1$. Thus, by lemma \ref{2l}, we have that
\begin{equation}\label{9c}
\sup_{a \in \mathbb{D}}\big|S_{\vec{u},\tau}(f_{a}(\sigma_{a})^{k})\big|\geq\big|
S^{n}_{\vec{u},\tau}(g_{n})(z)\big|=\bigg|\sum_{i=0}^{n}u_{i}(z)g_{n}^{(i)}
(\tau(z))\bigg|=\displaystyle\frac{n!|u_{n}(z)|}{\left(1-|\tau(z)|^{2}\right)
^{n+\gamma}}
\end{equation}
which implies that, $Q_{n}=\displaystyle\frac{|u_{n}(z)|}{\left(1-|\tau(z)|^{2}\right)
^{n+\gamma}}\in L^q(\mu)$.
Next, for $k=n-1$, we consider $g_{n-1}(z)=\big(f_{\tau(z)}\sigma_{\tau(z)}^{n-1}\big)
(z)$, $z\in \mathbb{D}$ such that $g_{n-1}^{(1)}
(\tau(z))=g_{n-1}^{(2)}
(\tau(z))=\cdots=g_{n-1}^{(n-2)}
(\tau(z))=0$ and $g_{n-1}^{(n-1)}
(\tau(z))=\displaystyle\frac{(n-1)!}{\left(1-|\tau(z)|^{2}\right)
^{n-1+\gamma}}$.  By lemma \ref{3l}, $\|g_{n-1}\|_{X}\leq 1$. Thus, by lemma \ref{2l}, we have that
\begin{eqnarray}\label{10c}
  \big|S^{n}_{\vec{u},\tau}(g_{n-1})(z)\big|&= & \bigg|\sum_{i=0}^{n}u_{i}(z)g_{n-1}^{(i)}(\tau(z))\bigg| \nonumber\\
  &\geq & \big|u_{n-1}(z)g_{n-1}^{(n-1)}(\tau(z))\big|-\big|u_{n}
  (z)g_{n-1}^{(n)}(\tau(z))\big|\nonumber\\
  &\geq & \frac{(n-1)!|u_{n-1}(z)|}{\left(1-|\tau(z)|^{2}\right)^
  {n-1+\gamma}}-\frac{C|u_{n}(z)|\|g_{n-1}\|_X}
  {\left(1-|\tau(z)|^{2}\right)^{n+\gamma}}.
\end{eqnarray}
By (\ref{9c}) and (\ref{10c}), we have that
\begin{eqnarray}\label{1d}
  Q_{n-1}&= &\frac{|u_{n-1}(z)|}{\left(1-|\tau(z)|^{2}\right)^
  {n-1+\gamma}}\nonumber\\
  &\lesssim & \frac{C|u_{n}(z)|\|g_{n-1}\|_X}
  {\left(1-|\tau(z)|^{2}\right)^{n+\gamma}}+\big|S^{n}
  _{\vec{u},\tau}(g_{n-1})(z)\big|\nonumber\\ &\leq &C Q_n+ \sup_{a\in\mathbb{D}}\big|S^{n}_{\vec{u},\tau}(f_{a}(\sigma_{a})^{n-1})\big|.
\end{eqnarray}
which implies that $ Q_{n-1}\in L^q(\mu)$.\\
Using  lemma \ref{2l}, (\ref{9c}) and (\ref{1d}), similarly we can proceed for $k=n-2$, and we find that
 \begin{equation}\label{2d}Q_{n-2}\in L^q(\mu)
 \end{equation}
 Further, for $j>k$ assume that
 \begin{equation}\label{3d}
   Q_{j}\in L^q(\mu)
 \end{equation}
For $k$, consider $g_{k}(z)=\big(f_{\tau(z)}\sigma_{\tau(z)}^{k}\big)(z)$ such that $g_{k}^{(1)}
(\tau(z))=g_{k}^{(2)}
(\tau(z))=\cdots=g_{k}^{(k-1)}
(\tau(z))=0$ and $g_{k}^{(k)}
(\tau(z))=\displaystyle\frac{k!}{\left(1-|\tau(z)|^{2}\right)
^{k+\gamma}}$. Thus, by lemma \ref{2l}, we have
\begin{eqnarray}\label{4d}
 \big|S^{n}_{\vec{u},\tau}(g_{k})(z)\big|&= & \big|\sum_{i=0}^{n}u_{i}(z)g_{k}^{(i)}(\tau(z))\big| \nonumber\\
  &\geq & \big|u_{k}(z)g_{k}^{(k)}(\tau(z))\big|-\sum_{i=k+1}^{n}\bigg|u_{i}(z)g_{k}^
  {(i)}(\tau(z))\big|\nonumber\\
  &\geq & \frac{k!|u_{k}(z)|}{\left(1-|\tau(z)|^{2}\right)^{k+
  \gamma}}-\sum_{i=k+1}^{n}\frac{C|u_{i}(z)|\|g_{k}\|_X}
  {\left(1-|\tau(z)|^{2}\right)^{i+\gamma}}.
\end{eqnarray}
Using (\ref{9c}), (\ref{1d}), (\ref{2d}), (\ref{3d}) and (\ref{4d}), we find that
\begin{eqnarray}
  Q_{k}&=&\frac{|u_{k}(z)|}{\left(1-|\tau(z)|^{2}\right)^{k+
  \gamma}}\nonumber\\
  &\lesssim & \sum_{i=k+1}^{n}\frac{C|u_{i}(z)|\|g_{k}\|_X}
  {\left(1-|\tau(z)|^{2}\right)^{i+\gamma}} +\big|S^{n}_{\vec{u},\tau}(g_{k})(z)\big|\nonumber\\
  &\leq &\sum_{i=k+1}^{n} C Q_i+ \sup_{a\in\mathbb{D}}\big|S^{n}_{\vec{u},\tau}(f_{a}(\sigma_{a})^{k})\big|\nonumber
\end{eqnarray}
which implies that $Q_k \in L^q(\mu)$.\\
$(c)\Rightarrow (b)$ This implication is trivial one.\\
$(b)\Rightarrow (a)$ Assume that $\sum_{k=0}^{n}Q_{k}=\displaystyle\sum_{k=0}^{n}\frac
{|u_{k}(z)|}{(1-|\tau(z)|^2)^{k+\gamma}}\in L^q(\mu)$.
 By lemma \ref{2l}, we have that
\begin{eqnarray}
 \big|S^{n}_{\vec{u},\tau}(f)(z)\big|&= & \bigg|\sum_{k=0}^{n}u_{k}(z)f^{(k)}(\tau(z))\bigg| \nonumber\\
  &\leq & \sum_{k=0}^{n}|u_{k}(z)|\big|f^{(k)}(\tau(z))\big| \nonumber\\
   &\leq &\sum_{k=0}^{n}\frac{C|u_{k}(z)|\|f\|_{X}}{\left(1-
   |\tau(z)|^{2}\right)^{k+\gamma}}\nonumber
\end{eqnarray}
This implies that $S^{n}_{\vec{u},\tau}:X \longrightarrow H_{\nu}^{\infty}$ is order bounded.\\
Proof of the part(B):\\
The proofs of the implications (iv)$\Rightarrow $ (iii)$\Rightarrow$ (ii)$\Rightarrow$ (i) are the same as the proof of part (A). The proof of the implication (i)$\Rightarrow$ (iv) follows from lemma \ref{3l}.
This completes the proof of the theorem \ref{3t}.
\end{proof}

\subsection*{Acknowledgment}
We would like to thank the referees for several helpful comments and suggestions.\\
The  first author is thankful to NBHM(DAE)(India) for the research project  (Grant No. 02011/30/2017/R\&{D} II/12565).\\ The second author is thankful to  CSIR (India) for the Junior Research Fellowship (File no. 09/1231(0001)/2019-EMR-I).

\end{document}